\documentclass[10pt]{amsart}
\usepackage{amsmath,amssymb,amsthm,amsfonts,mathrsfs,amsopn}
\usepackage[all]{xy}
\usepackage{CJK}
\usepackage{dsfont}
\usepackage{hyperref}  
\usepackage{color}
\usepackage{esint}
\usepackage{mathtools}
\mathtoolsset{showonlyrefs}
\usepackage{slashed}

\usepackage{graphicx}

\usepackage{color}  
\definecolor{purple}{rgb}{1.0, 0.0, 1.0}

\usepackage[margin=1.1in]{geometry}

\newcommand{\p}{\partial}
\newcommand{\R}{\mathbb{R}}
\newcommand{\C}{\mathbb{C}}

\newcommand{\D}{\slashed{D}}

\newcommand{\dd}{\mathop{}\!\mathrm{d}}
\newcommand{\dv}{\dd{v}}
\newcommand{\genus}{\gamma}
\newcommand{\cliff}{\mathfrak{m}}

\DeclareMathOperator{\Spin}{Spin}
\DeclareMathOperator{\SO}{SO}
\DeclareMathOperator{\Spec}{Spec}

\DeclareMathOperator{\End}{End}
\DeclareMathOperator{\Id}{Id}

\numberwithin{equation}{section}

\newtheorem{thm}{Theorem}[section]

\newtheorem{lemma}[thm]{Lemma}
\newtheorem{prop}[thm]{Proposition}

\newtheorem{rmk}{Remark}[section]

\title[Super-Liouville equation]{Existence results for a super-Liouville equation \\on compact surfaces}

\author[A. Jevnikar, A. Malchiodi, R. Wu]{Aleks Jevnikar, Andrea Malchiodi, Ruijun Wu}
\address{Scuola Normale Superiore, Piazza dei Cavalieri 7, 56126 Pisa, Italy.}
\email{aleks.jevnikar@sns.it, andrea.malchiodi@sns.it, ruijun.wu@sns.it}

\begin{document}

\begin{abstract}
 We are concerned with a super-Liouville equation on compact surfaces with  genus larger than one, 
   obtaining the first non-trivial existence result for this class of problems via min-max methods.
  In particular we make use of a Nehari manifold and, 
 after showing the validity of the Palais-Smale condition,  
  we exhibit either a mountain pass or linking geometry. 

\end{abstract}

\maketitle

{\footnotesize

\medskip

}

\

\section{introduction}

The Liouville equation in two dimensions, which has the form
\begin{equation}\label{eq:prescribing curvature equation}
 -\Delta u= \widetilde{K} e^{2u}-K,
\end{equation}
for some given functions~$K,\widetilde{K}$   on a  surface~$M$, has been extensively studied and has wide applications in geometry and physics. 
A typical example is the prescription of curvature.
Let~$g$ be a Riemannian metric on a surface~$M$ with Gaussian curvature~$K=K_g$ and let~$\widetilde{K}$ be a given function on~$M$. 
The question is whether there exists a functions~$u \in C^\infty(M)$ such that the conformal metric~$\widetilde{g}=e^{2u}g$ has Gaussian curvature~$\widetilde{K}$, 
see e.g. \cite{changyang1987prescribing, kazdan1975scalar}. 

Since the Gaussian curvature for~$\widetilde{g}$ is given by~$e^{-2u}(K_g-\Delta_g u)$, the problem is equivalent to the solvability of 
equation~\eqref{eq:prescribing curvature equation}. 
Observe that the conformal factor~$u$ within the conformal class of~$[g]$ can be found as a critical point of the following functional:
\begin{equation}
I(u):=\int_M \left( |\nabla u|^2+2K_gu-\widetilde{K}e^{2u} \right) \dv_g.
\end{equation}

When~$M$ is a closed Riemann surface, which is the case we are interested in for this paper, 
the function~$\widetilde{K}$ has to satisfy the Gauss--Bonnet formula with respect to the new metric~$\widetilde{g}$.
When~$\widetilde{K}$ is constant with the sign compatible with the Gauss--Bonnet formula, the equation is always solvable, according to the~\emph{uniformization theorem}. 
For non-constant~$\widetilde{K}$, though not being totally solved, we have a good understanding of the problem in most cases, see e.g.~\cite[Chapter 5]{schoenyau1994lectureOnDG-I} and ~\cite[Chapter 6]{aubin1998somenonlinear}. 

\medskip

More recently, equation~\eqref{eq:prescribing curvature equation} has been studied in the context of hyperelliptic curves and of the Painlev\'e equations, see \cite{chai-lin-wang} and \cite{chen-kuo-lin}, respectively. 

\

Equation~\eqref{eq:prescribing curvature equation} plays also an important role in mathematical physics. On one hand, it arises in Electroweak and Chern-Simons self-dual vortices, see \cite{spruck-yang,tarantello,yang}. On the other hand, it appears in the Liouville field theory with applications to  string theory, see \cite{ops1988extremals,polya1,polya2}. See also \cite{hawking} for a recent connection between \eqref{eq:prescribing curvature equation} and the Hawking mass.
 
\

Motivated by the supersymmetric extension of the Liouville theory, the authors in~\cite{jost2007superLiouville} 
introduced the following so-called \emph{super-Liouville functional}:
\begin{equation}
 \widetilde I(u,\psi):=\int_M \biggr( |\nabla u|^2+ 2K_g u-e^{2u} +2\left<(\D+e^u)\psi,\psi\right> \biggr) \dv_g,  
\end{equation}
where $\D$ is the Dirac operator acting on spinors $\psi$, see Subsection~\ref{subsec:spinor} for precise definitions. In a series of works they  performed blow-up analysis and studied the compactness of the solution spaces under weak assumptions and in various setting; see e.g.~\cite{jost2007superLiouville, jost2009energy,jost2014qualitative,jost2015LocalEstimate} and the references therein. For the role of the super-Liouville equations in physics we refer to \cite{super1,super2,super3}. One should note that the sign conventions adopted above are adapted to the sphere case. 

In this paper we consider the problem posed on a closed Riemann surface of genus~$\genus> 1$. 
In this case the coefficients in the action functional need to be adapted to the Gauss--Bonnet formula. 
Let~$g$ be a Riemannian metric compatible with the given complex structure.
We are going to consider the following functional: 
\begin{equation}\label{eq:super-Liouville functional}
 J_\rho(u,\psi)\coloneqq \int_M \biggr( |\nabla u|^2+2K_gu+e^{2u}
            +2\left<(\D-\rho e^u)\psi,\psi\right> \biggr) \dv_g,
\end{equation}
where the parameter~$\rho$ is a positive constant. 
We are adopting a different notation from that in~\cite{jost2007superLiouville}, making  our choice  compatible with  equation~\eqref{eq:prescribing curvature equation}.
The Euler--Lagrange equation for~$J_\rho$ is 
\begin{equation}\label{eq:EL for super-Liouville}
\tag{EL}
 \begin{cases}
  \Delta_g u{}=&{}e^{2u}+K_g-\rho e^u|\psi|^2, \vspace{0.2cm}\\
  \D_g\psi{}=&{}\rho e^u\psi,
 \end{cases}
\end{equation}
which takes the name of \emph{super-Liouville equations}. The system~\eqref{eq:EL for super-Liouville} clearly admits the {\em trivial solution} $(u_*,0)$, where~$u_*$ satisfies 
\begin{equation}
 -\Delta u_*= -e^{2u_*}-K_g
\end{equation}
and whose existence is given by the uniformization theorem.  
This is also a~\emph{ground state solution}  in the sense that it has minimal critical level: this follows from the fact that  the spinorial part does not affect the critical levels, while the scalar component of the functional is coercive and convex. The latter properties 
also yield uniqueness of such a trivial solution. 
The aim of the present paper is to find a solution with non-zero spinor part, a so-called \emph{non-trivial solution}. 

\

\textbf{Conformal symmetry and reduction to uniformized case.}
System~\eqref{eq:EL for super-Liouville} admits a conformal symmetry in the following sense. 
Suppose that~$(u,\psi)$ is a solution of~\eqref{eq:EL for super-Liouville}, 
let~$v\in C^\infty(M)$ and consider the metric~$\widetilde{g}\coloneqq e^{2v}g$.  
There exists an isometric isomorphism~$\beta\colon S_g\to \widetilde{S}_{\widetilde{g}}$ of the spinor bundles corresponding to different metrics such that
\begin{equation}\label{eq:D-conf}
 \widetilde{\D}_{\widetilde{g}}\left(e^{-\frac{v}{2}}\beta(\psi)\right)
 =e^{-\frac{3}{2}v}\beta(\D_g\psi),
\end{equation}
see e.g.~\cite{ginoux2009dirac, hitchin1974harmonicspinors}, where we are using the notation from~\cite{jost2018symmetries}.
Thus the pair
$$
 \begin{cases}
  \widetilde{u}=u-v, \vspace{0.2cm}\\
  \widetilde{\psi}= e^{-\frac{u}{2}}\beta(\psi),
 \end{cases}
$$
solves the system 
\begin{align}
 \Delta_{\widetilde{g}}\widetilde{u} 
 ={}&e^{-2v}\Delta_g(u-v)= e^{-2v}(e^{2u}+K_g-\rho e^{u}|\psi|^2-\Delta_g v) \\
 ={}&e^{2(u-v)}+e^{-2v}(K_g-\Delta_g v)-\rho e^{u-v}|e^{-\frac{v}{2}} \beta(\psi)|^2 \\
 ={}&e^{2\widetilde{u}}+{K}_{\widetilde{g}}
 -\rho e^{\widetilde{u}}|\widetilde{\psi}|^2, \\ 
\widetilde{\D}_{\widetilde{g}} \widetilde{\psi}
 ={}&\rho e^{-\frac{3}{2}v}\beta(e^u\psi)
     =\rho e^{u-v}\left(e^{-\frac{1}{2}v}\beta(\psi)\right)
     =\rho e^{\widetilde{u}}\widetilde{\psi},  
\end{align}
analogous to \eqref{eq:EL for super-Liouville}. 
Therefore, we can work with a convenient  background metric inside the given conformal class.
W.l.o.g., recalling that the genus is larger than one, we assume that the background metric~$g_0$ is uniformized, meaning that~$K_{g_0}\equiv -1$: notice that 
such a metric is unique. 
In this case the trivial solution is given by~$(0,0)$: 
the main result of the paper is the existence of a non-trivial min-max solution obtained via a variational approach. 

\begin{thm} \label{thm}
 Let~$M$ be a closed Riemann surface of genus~$\genus>1$ with Riemannian metric~$g$.
 Let~$g_0\in [g]$ be a conformal uniformized metric, i.e.~$K_{g_0}\equiv -1$, and suppose that the spin structure is chosen so that~$0\notin  \Spec(\D_{g_0})$.
 Then for any~$\rho\notin\Spec(\D_{g_0})$, there exists a non-trivial solution to~\eqref{eq:EL for super-Liouville}.   
\end{thm}

We stress that this is the first non-trivial existence result for this class of problems. Moreover, observe that by \eqref{eq:D-conf} ~$\dim\ker(\D_g)$ is a conformal invariant, and the condition~$0\notin \Spec(\D_{[g]})$ is valid for many spin structures and conformal structures, as it will be explained later. 

\begin{rmk}
 Note that the spinor bundle~$S\to M$ admits global automorphisms, e.g. the quaternionic structures, which form a group.
 These are parallel with respect to~$\nabla^s$ and commute with the Clifford multiplications by tangent vectors, see~\cite[Sect. 2]{jost2016regularity}.
 The functional~$J_\rho$ is thus invariant under the actions of such isometries. 
It follows that there exist more than one non-trivial solution (at least eight, which is the cardinality of  the quaternion group). 
Given a solution ~$(U,\Psi)$, an intuitive example is the antipodal solution~$(U,-\Psi)$, which is in the orbit of the quaternionic structure group actions.  
\end{rmk}

Concerning the case of genus one, i.e. when the base surface is a torus, the problem might not be well-defined. 
Indeed, if we take~$\widetilde{K}$ and~$K$ to be zero, then the system~\eqref{eq:EL for super-Liouville} has only trivial solutions of the form~$(a,0)$ where~$a\in\R$. 
Meanwhile in the sphere case, where both~$\widetilde{K}$ and~$K$ should be~$1$, the functional turns out to be even more strongly indefinite, and   admits 
neither the classical mountain pass nor the linking geometry. 
In the genus-one case it might be interesting to consider the case of changing-sign $\tilde{K}$, as it was done in 
\cite{kazdan1975scalar} for the prescribed Gaussian curvature problem. 

\medskip

The main difficulty in studying \eqref{eq:EL for super-Liouville}  is that the Dirac operator is strongly indefinite: the spectrum of~$\D$ is real and symmetric with respect to the origin. 
The classical theory for variational problems involving Laplacians or Schr\"odinger operators, where the positive parts usually dominates the behavior of the functional, fails to work for Dirac type actions.
There were methods developed for general strongly indefinite variational problems, see e.g.~\cite{benci1982oncritical, benci1979critical, hulshof1993differential}, but they are not directly applicable to Dirac operators.
Dirac operators usually relates more closely to the geometry and topology of the spin manifolds.
Recently several attempts have been made to attack such problems.
With suitable nonlinearities as perturbation adding to the geometric equations, T. Isobe made remarkable progress in adapting the classical theory of calculus of variations to the Dirac  setting~\cite{isobe2010existence,isobe2011nonlinear, isobe2019onthemultiple}.    
Combined with the methods of Robinowitz-Floer homology, A. Maalaoui and V. Martino  also obtained existence results of some nonlinear Dirac type equations, see~\cite{maalaoui2013rabinowitz,maalaoui2015therabinowitz,maalaoui2017characterization} and the references therein.
In the case of super-Liouville equations we have to deal with an exponential nonlinearity, which does not fit in the above settings. 
Moreover, we are directly facing a geometric problem without auxiliary nonlinear perturbations, which is usually harder to deal with. 

\

The article is organized in the following way.
In the second section we introduce some preliminaries in spin geometry and discuss existence of harmonic 
spinors depending on the genus and on the conformal class. We also introduce suitable Sobolev spaces to work 
with and the Moser-Trudinger inequality. 
In the third section we tackle the strong-indefiniteness of the functional by building a natural constraint which defines a generalized Nehari manifold~$N$. 
We then verify the Palais--Smale condition for~$J_\rho|_N$ by showing first some a-priori bounds and then 
proving strong subsequential convergence. 
For suitable $\rho$ we finally show either mountain pass or linking geometry on the Nehari manifold which yield
the existence of a min-max critical point for~$J_\rho$: the 
details of this construction are given in the last section. 

\

\noindent {\bf Acknowledgments.}
A.M. has been partially supported by the projects {\em Geometric Variational Problems} and {\em Finanziamento a supporto della ricerca di base} from Scuola Normale Superiore. 
A.J. and A.M. has been partially supported by MIUR Bando PRIN 2015 2015KB9WPT$_{001}$. They are also members of GNAMPA as part of INdAM.
A.J. and R.W. are supported by the  Centro  di  Ricerca  Matematica  ‘Ennio  de  Giorgi’.

\

%%%%%%%%%%
\section{Preliminaries}

We will assume some background in spin geometry and Sobolev spaces.
For detailed material one can refer to~\cite{ammann2003habilitation, ginoux2009dirac, gilbarg2001elliptic, lawson1989spin}. 

\

\subsection{Spinor bundles and Dirac operator} \label{subsec:spinor}
Here we introduce our setting and fix the notation.  
Let~$M$ be a closed Riemann surface with a fixed conformal structure and of genus~$\genus$.  
Let~$g$ be a Riemannian metric in the given conformal class and denote the Gaussian curvature by~$K_g$. 
The orthonormal frame bundle~$P_{\SO}(M,g)\to M$ is then a principal~$\SO(2)$ bundle.
Let~$\Spin(2)= U(1)\to \SO(2)$ be the two-fold covering of the circle. 
A \emph{spin structure} is given by a principal~$\Spin(2)$ bundle~$P_{\Spin}(M,g)\to M$ together with an equivariant two-fold covering
\begin{equation}
 P_{\Spin}(M,g)\to P_{\SO}(M,g).
\end{equation}
In dimension two such double coverings always exist; moreover they are in one-to-one correspondence with the elements in~$H^1(M;\mathbb{Z}_2)$, see e.g. \cite[Chapter 2]{lawson1989spin}.
This cohomology group has cardinality~$2^{2\genus}$. 

Let~$S\equiv S_g\to M$ be the associated spinor bundle with a real Riemannian structure~$g^s$ and induced spin connection~$\nabla^s$:
sections of~$S$ are called~\emph{spinors}.
Recall that the \emph{Dirac operator}~$\D$ acting on spinors is defined as the composition of the following chain
\begin{equation}
 \Gamma(S)\xrightarrow{\nabla^s} \Gamma(T^*M\otimes S)\xrightarrow{\cong} \Gamma(TM\otimes S)\xrightarrow{\cliff}\Gamma(S),
\end{equation}
where the second isometric isomorphism is given by the identification via the metric~$g$,  the third arrow~$\cliff$ denotes the Clifford multiplication, and the~$\End(S)$-valued map~$\cliff\colon TM\to \End(S)$ satisfies the following Clifford relation:
\begin{equation}
 \cliff(X)\cliff(Y)+\cliff(Y)\cliff(X)=-2g(X,Y), \qquad \forall X,Y\in\Gamma(TM).
\end{equation}
Later, for simplicity, we will write~$X\cdot\psi$ for~$\cliff(X)\psi$, where~$X\in\Gamma(TM)$ and~$\psi\in\Gamma(S)$.
In terms of a local orthonormal frame~$(e_i)_{i=1,2}$ we then have the {\em Dirac operator}
\begin{equation}
 \D\psi=\sum_{i}\cliff(e_i)\nabla^s_{e_i}\psi,\quad \forall \psi\in \Gamma(S).
\end{equation}
This is a self-adjoint elliptic operator of first order, and it has a finite-dimensional kernel consisting of \emph{harmonic spinors}.
The dimension of the space of harmonic spinors is a conformal invariant, but it depends on the choice of spin structures and the conformal structures in general.
The Bochner-Lichnerowicz formula
\begin{equation}
 \D^2=(\nabla^s)^* \nabla^s+\frac{Scal}{4}
\end{equation}
implies that there is no non-trivial harmonic spinor if~$Scal\ge0$ and~$Scal\not\equiv 0$. 
In particular, there is no harmonic spinor on the 2-sphere with arbitrary metric (since there is only one conformal structure on the 2-sphere).
However, when the genus~$\genus$ is greater than or equal to $1$, there might exist non-trivial harmonic spinors for some choice of spin structures. 
The dimensions of the spaces of harmonic spinors have been computed in literature e.g.~\cite{hitchin1974harmonicspinors, bar1992harmonic, bores1994harmonic}. 
We summarize some facts here to have a picture of the different cases.

\

\subsection{Examples of Riemann surfaces with no non-trivial harmonic spinors}
Here we give some examples of Riemann surfaces having negative Euler characteristic~$2\pi\chi(M)=4\pi(1-\genus)<0$ but admitting no non-trivial harmonic spinors. 

Any element~$\alpha\in H^1(M,\mathbb{Z}_2)$ determines a spin structure~$\xi(\alpha)$, as well as a holomorphic line bundle~$\mathscr{L}_\alpha$ such that~$\mathscr{L}_\alpha\otimes_\C \mathscr{L}_\alpha= \mathscr{K}_M$, where~$\mathscr{K}_M$ denotes the canonical line bundle of~$M$, see e.g. \cite{hitchin1974harmonicspinors, lawson1989spin}.
Denote by~$\mathscr{O}(\mathscr{L_\alpha})$ the sheaf of germs of holomorphic sections of the holomorphic line bundle~$\mathscr{L}_\alpha$, and set~$h_{\alpha,g}^0 =\dim H^0(M, \mathscr{O}(\mathscr{L}_\alpha))$. 
If the associated spinor bundle~$S\equiv S(\alpha,g)$ admits a space of harmonic spinors of dimension~$h_{\xi(\alpha),g}$, then
\begin{equation}
 h_{\xi(\alpha),g}=2h^0_{\alpha,g}. 
\end{equation}
It is known that, for a Riemann surface~$M$ of genus~$\genus$, there are precisely~$2^{\genus-1}(2^\genus+1)$ spin structures~$\alpha$ on~$M$ for which ~$h^0_{\alpha,g}$ is an even number (such spin structures are called \emph{even spin structures} on~$M$), and for the other~$2^{\genus-1}(2^{\genus}-1)$ spin structures the number~$h^0_{\alpha,g}$ is odd (\emph{odd spin structures}). 
 
For~$\genus=1$ ~$M$ is topologically a torus, and for any conformal structure~$[g]$ we have four spin structures: three even spin structures with no non-trivial harmonic spinors and one odd spin structure (the trivial one~$\alpha= 0$) with one-dimensional space of positive harmonic spinors (hence~$h_{\xi(0);g}=2$). 

For~$\genus=2$ the description is similar, namely for any conformal structure~$[g]$ there are ten even spin structures with no non-trivial harmonic spinors and six odd spin structures with one-dimensional space of positive harmonic spinors (hence~$h_{\xi(0)}=2$).

These are the known cases where the dimension of~$\ker(D)$ is independent of the choice of metric~$g$ (i.e. the choice of the Riemann surface structure on~$M$). 
When the genera become larger, the dimension of the kernels generally depends on the conformal class.
Even in this case we still have  many examples where there are no non-trivial harmonic spinors.

Recall that a hyperelliptic Riemann surface is a complex projective curve admitting a rational surjective map onto~$\C P^1$ which is 2-to-1 up to a finite set of branching points.
All Riemann surfaces of genera~$\genus\le 2$ are hyperelliptic, while there exist non-hyperelliptic surfaces of all genera~$\genus\ge 3$. 

For the hyperelliptic case, C. B\"ar~\cite{bar1992harmonic} showed that the spin structures correspond one-to-one to the pairwise inequivalent square roots of the canonical divisor, and in terms of a suitably defined \emph{weight} of the divisors, he also clarified the dimensions~$h^0$ of the kernels: 
\begin{enumerate}
 \item if~$M$ is hyperelliptic with~$\genus=2k+1$, 
       \begin{itemize}
        \item there is exactly one spin structure of weight~$\genus-1$ and in this case~$h^0=\frac{\genus+1}{2}=k+1$;
        \item for~$w=1,3,5,\cdots,\genus-2$, there are exactly~$\binom{2\genus+2}{\genus-w}$ spin structures of weight~$w$ and in this case~$h^0=\frac{w+1}{2}$;
        \item there are exactly~$\binom{2\genus+1}{\genus}$ spin structures of weight~$-1$ and in this case~$h^0=0$;
       \end{itemize}
  \item if~$M$ is hyperelliptic with~$\genus=2k$,
        \begin{itemize}
         \item there is exactly~$2\genus+2$ spin structure of weight~$\genus-1$ and in this case~$h^0=[\frac{\genus+1}{2}]=k$;
        \item for~$w=1,3,5,\cdots,\genus-1$, there are exactly~$\binom{2\genus+2}{\genus-w}$ spin structures of weight~$w$ and in this case~$h^0=\frac{w+1}{2}$;
        \item there are exactly~$\binom{2\genus+1}{\genus}$ spin structures of weight~$-1$ and in this case~$h^0=0$;
        \end{itemize}

\end{enumerate}
For  non-hyperelliptic surfaces, there are also known examples where the dimensions of kernels are computed. 

For a genus~$\genus=3$ non-hyperelliptic surface, among the~$2^{2\genus}=64$ spin structures there are~$28$ odd ones with~$h^0=1$ and~$36$ even ones with~$h^0=0$. 

The case for~$\genus=4$ non-hyperelliptic surfaces is different: there are in total~$2^{2\genus}= 256$ spin structures, 120 of them are odd with~$h^0=1$, and for the other~$136$ even spin structures, one of the followings may happen:
\begin{enumerate}
 \item[(I)] there exists a unique even spin structure with~$h^0=2$, while the other~$135$ even spin structures have~$h^0=0$;
 \item[(II)] all the 136 spin structures have~$h^0=0$.
\end{enumerate}
A non-hyperelliptic Riemann surface is called of type~(I) or~(II) if  it satisfies the corresponding above conditions. 
Both classes are non-empty. 

\

\subsection{Sobolev spaces for spinors}

The spinor bundle~$S=S_g$ has a Riemannian structure~$g^s$ and a spin connection~$\nabla^s$ induced from the Levi-Civita connection.
Then we can define the usual Sobolev spaces with integer differentiability, namely ~$W^{k,p}(S)$ consists of the spinors whose~$k$-th covariant derivatives are in~$L^p$ for~$k\in \mathbb{N}$ and~$p\in [1,+\infty]$ and~$W^{-k,q}(S)\coloneqq (W^{k,p}(S))^*$ where~$q$ is the H\"older conjugate of~$p$. 
Here we will consider also fractional Sobolev exponents, see the discussion in the sequel.  

Recall that~$\D=\D_g$ is a first order elliptic operator which is essentially self-adjoint. 
The spectrum~$\Spec(\D)$ is discrete and consists of real eigenvalues,~$\Spec_0(\D)\cup\{\lambda_k\}_{k\in\mathbb{Z}\backslash\{0\}}$, where~$\Spec_0(\D)$ stands for the  zero element in the spectrum (or the empty set) while the lambda's are the non-zero eigenvalues, indexed by~$\mathbb{Z}_*\equiv \mathbb{Z}\backslash\{0\}$ in an increasing order (in absolute value) and counted with multiplicities:
\begin{equation}
 -\infty \leftarrow\cdots\le \lambda_{-l-1}\le \lambda_{-l}\le\cdots\le \lambda_{-1}\le 0 
 \le \lambda_1\le \cdots \le \lambda_k \le \lambda_{k+1}\le \cdots \to +\infty.
\end{equation}
Moreover, the spectrum is symmetric with respect the the origin when~$\dim M=2$. 
Let~$\varphi_k$ be the eigenspinor corresponding to~$\lambda_k$, ~$k\in\mathbb{Z}_*$ with~$\|\varphi_k\|_{L^2(M)}=1$, and let ~$\varphi_{_{0,j}}$, ~$1\le j\le h^0$, be an orthonormal basis of~$\ker(\D)$. 
These together form a complete orthonormal basis of~$L^2(S)$: 
 any spinor~$\psi\in\Gamma(S)$ can be  expressed in terms of this basis as 
\begin{equation}\label{eq:spinor in basis}
 \psi=\sum_{k\in \mathbb{Z}_*}a_k\varphi_k
   +\sum_{1\le j\le h^0} a_{0,j}\varphi_{_{0,j}},  
\end{equation}
and the Dirac operator acts as
\begin{equation}
 \D\psi= \sum_{k\in\mathbb{Z}_*} \lambda_k a_k \varphi_k.   
\end{equation}
For any~$s>0$, the operator~$|\D|^s\colon \Gamma(S)\to \Gamma(S)$ is defined as 
\begin{equation}
 |\D|^s\psi =\sum_{k\in\mathbb{Z}_*} |\lambda_k|^s a_k\varphi_k,
\end{equation}
provided that the right-hand side belongs to~$L^2(S)$. The domain of~$|\D|^s$ is 
\begin{equation}
 H^s(S)\coloneqq \left\{\psi\in L^2(S)\mid \int_M\left<|\D|^s\psi,|\D|^s\psi\right>\dv_g <\infty  \right\}, 
\end{equation}
which is a Hilbert space with inner product
\begin{equation}
 \left<\psi,\phi\right>_{H^s} 
 = \left<\psi,\phi\right>_{L^2}
 +\left<|\D|^s\psi, |\D|^s\phi\right>_{L^2}.
\end{equation}
For~$s=k\in\mathbb{N}$,~$H^k(S)=W^{k,2}(S)$ and the above norm is equivalent to the Sobolev~$W^{k,2}$ norm.
For~$s<0$,~$H^{s}(S)$ is by definition the dual space of~$H^{-s}(S)$. 

Since~$S$ has finite rank, the general theory for Sobolev's embedding on closed manifold continues to hold here.
In particular, for~$0<s<1$ and~$q\le \frac{2}{1-s}$, we have the continuous embeddings 
\begin{equation}
 H^s(S) \hookrightarrow L^{q}(S). 
\end{equation}
Furthermore, for~$q<\frac{2}{1-s}$ the embedding is compact, 
 see e.g. \cite{ammann2003habilitation} for more details. 

We will mainly be interested in the case~$s=\frac{1}{2}$, for which~$\frac{2}{1-s}=4$.
This is the largest space on which the Dirac action of the form
\begin{equation}
 \psi\mapsto \int_M \left<\D\psi,\psi\right>\dv_g
\end{equation}
is well-defined. 
Note that for~$\psi\in H^{\frac{1}{2}}(S)$ we have~$\D\psi\in H^{-\frac{1}{2}}(S)$, which is defined in the distributional sense. 
Thus we can define the duality pairing
\begin{equation}
 \left<\D\psi,\psi\right>_{H^{-\frac{1}{2}}\times H^{\frac{1}{2}}} \in\R.
\end{equation}
On the other hand, by the expression~\eqref{eq:spinor in basis} we see that the function
\begin{equation}
 g^s_x \left(\D\psi(x),\psi(x)\right)
\end{equation}
is in~$L^1(M)$, whose integral is exactly given by~$\sum_{k\in\mathbb{Z}_*} \lambda_k a_k^2 <\infty$.
By this we validate the Dirac action in the equivalent form
\begin{equation}
 \left<\D\psi,\psi\right>_{H^{-\frac{1}{2}}\times H^{\frac{1}{2}}}
 =\int_M \left<\D\psi(x),\psi(x)\right>_{g^s(x)}\dv_g(x). 
\end{equation}

Suppose~$h^0=0$, i.e. there are no non-trivial harmonic spinors.
Then the Dirac operator~$\D$ is invertible.  
Splitting into the positive and negative parts of the spectrum~$\Spec(\D)$, we have the decomposition 
\begin{equation}\label{eq:split}
 H^{\frac{1}{2}}(S)= H^{\frac{1}{2},+}(S)\oplus H^{\frac{1}{2},-}(S).
\end{equation}
Let~$\psi=\psi^+ +\psi^-$ be decomposed accordingly: then, 
\begin{align}
 \int_M\left<\D\psi^+,\psi^+\right>\dv_g 
 =\int_M\left<|\D|^{\frac{1}{2}}\psi^+,|\D|^{\frac{1}{2}}\psi^+\right>\dv_g \ge \lambda_1(\D_g)\|\psi^+\|_{L^2(M)}^2, 
\end{align}
where~$\lambda_1$ is the first positive eigenvalue of~$\D=\D_g$.
Hence
\begin{align}
 \|\psi^+\|^2_{H^{\frac{1}{2}}}
 =&\|\psi^+\|^2_{L^2}+\||\D|^{\frac{1}{2}}\psi^+\|_{L^2}^2 \\  
 \le&(\lambda_1(\D_g)+1)\||\D|^{\frac{1}{2}}\psi^+\|^2_{L^2}
 \le(\lambda_1(\D_g)+1)\|\psi^+\|^2_{H^{\frac{1}{2}}}. 
\end{align}
That is, for a given~$g$, the integral~$\int_M \left<\D\psi^+,\psi^+\right>\dv_g$ defines a norm on~$H^{\frac{1}{2},+}(S)$ equivalent to the Hilbert's. 
Similarly, on~$H^{\frac{1}{2},-}(S)$ there is an equivalent norm given by
\begin{equation}
 -\int_M \left<\D\psi^-,\psi^-\right>\dv_g
 =\||\D|^{\frac{1}{2}}\psi^-\|^2_{L^2}.
\end{equation}
Consequently,
\begin{equation}
 \int_M \left[ \left<\D\psi^+,\psi^+\right>-\left<\D\psi^-,\psi^-\right> \right] \dv_g
 =\||\D|^{\frac{1}{2}}\psi^+\|^2_{L^2}
 +\||\D|^{\frac{1}{2}}\psi^-\|^2_{L^2}
\end{equation}
defines a norm equivalent to the~$H^{\frac{1}{2}}$-norm. 

\

\subsection{Moser--Trudinger embedding}
Another space we use frequently is~$H^1(M)=W^{1,2}(M,\R)$. 
Consider the subspace in $H^1(M)$ of the functions with zero average 
\begin{equation}
 H^1_0(M)\coloneqq \left\{u\in H^1(M)\mid \int_M u\dv_g=0\right\}.
\end{equation}
Then~$H^1(M)= \R\oplus H^1_0(M)$, and any $u \in H^1(M)$ can be written as~$u=\bar{u}+\widehat{u}$ where~$\bar{u}=\fint_M u\dv_g $ denotes the average of~$u$. By Poincar\'e's inequality,~$\|\nabla \widehat{u}\|_{L^2}$ defines a norm equivalent to~$\|\widehat{u}\|_{H^1}$ on~$H^1_0(M)$, and 
\begin{equation}
 |\bar{u}|+\|\nabla\widehat{u}\|_{L^2}
\end{equation}
 a norm equivalent to~$\|u\|_{H^1}$. 
The Sobolev embedding theorems imply that for any~$p<\infty$, ~$H^1(M)$ embeds into~$L^p(M)$ continuously and compactly. 
Furthermore, the Moser--Trudinger inequality states that there exists~$C>0$ such that 
\begin{equation}
 \int_M \exp\left(\frac{4\pi|\widehat{u}|^2}{\|\nabla\widehat{u}\|^2_{L^2(M)}}\right)\dv_g\le C. 
\end{equation}
As a consequence 
\begin{equation}
 8\pi\log\int_M e^{\widehat{u}}\dv_g
 \le\frac{1}{2}\int_M |\nabla\widehat{u}|^2\dv_g+ C.
\end{equation}
This implies that~$e^{u}$ is~$L^p$ integrable for any~$p> 0$. 
Moreover, the map
\begin{equation}
 H^1(M)\ni u\mapsto e^u\in L^1(M)
\end{equation}
is compact (see e.g.~\cite[Theorem 2.46]{aubin1998somenonlinear}). 
It follows that the maps~$H^1(M)\ni u\mapsto e^u\in L^p(M)$ are compact for all~$p > 0$.  

%%%%%%%%%%%
\section{A natural constraint and the Palais-Smale condition}

It is standard to prove that the functional
\begin{equation}
 J_\rho\colon H^1(M)\times H^{\frac{1}{2}}(S)\to \R
\end{equation}
defined in formula \eqref{eq:super-Liouville functional}  is  of class~$C^1$. 
The critical points of~$J_\rho$, which are weak solutions of~\eqref{eq:EL for super-Liouville}, are actually smooth.
To see this we can use the argument from~\cite{jost2007superLiouville}.
Note that, although the authors there are using different Banach spaces, the proof goes quite similarly and is omitted here.
Alternatively, note that~$u\in H^1(M)$ implies~$e^{u}\in L^p(M)$ for any~$p<\infty$,i.e. the equation is actually subcritical and we can appeal to a bootstrap argument to obtain the full regularity. 
 
To obtain a non-trivial solution to the system~\eqref{eq:EL for super-Liouville} we employ a min-max approach.
As observed, thanks to the conformal covariance of the system, it is sufficient to consider the uniformized metric. 
From now on we assume that~$g$ has constant Gaussian curvature~$K\equiv-1$. 
For this choice we then look for non-trivial critical points of the functional
\begin{equation}\label{eq:functional-uniformized metric}
 J_\rho(u,\psi)
 =\int_M \biggr( |\nabla u|^2-2u+e^{2u}+2\left(\left<\D\psi,\psi\right>-\rho e^u|\psi|^2\right) \biggr)\dv_g,
\end{equation}
which are non-trivial solutions of the system
\begin{equation}\label{eq:EL-uniformized metric}\tag{$EL_0$} 
 \begin{cases}
  \Delta_g u=e^{2u}-1-\rho e^u|\psi|^2, \vspace{0.2cm}\\
  \D_g\psi= \rho e^u\psi.     
 \end{cases} 
\end{equation}
The argument in the sequel is simplified by this assumption, but it can be modified and adapted to a general metric. 
Note that in the uniformized case the Gauss-Bonnet formula yields
\begin{equation}
 vol(M,g)=-2\pi\chi(M)=4\pi(\genus-1). 
\end{equation}

Observe that in the functional~$J_\rho$ the first part is coercive and convex. The main difficulty is due to the spinorial part which is strongly indefinite. 
To overcome this issue we are inspired by an idea from~\cite{maalaoui2017characterization} and we consider a natural constraint: in the next section we will find critical points of the restricted functional. 

\

\subsection{A Nehari type manifold}
Roughly speaking, the space~$H^{\frac{1}{2},-}(S)$ defined in \eqref{eq:split} contains infinitely-many directions decreasing the functional $J_\rho$ to negative infinity and the usual variational approaches can not be applied. 
Hence we introduce a natural constraint in order to exclude most of these directions, obtaining a submanifold in~$H^1(M)\times H^{\frac{1}{2}}(S)$, which we still call it a \emph{Nehari manifold}, though it may not fit into the classical definition as in~\cite{ambrosetti2007nonlinear}.
This may be considered to be a Nehari manifold in the generalized sense, as in~\cite{pankov2005periodic, szulkin2009ground, szulkin2010themethod}.  

Let~$P^\pm\colon H^{\frac{1}{2}}(S)\to H^{\frac{1}{2},\pm}(S)$ be the orthonormal projection according to the splitting in \eqref{eq:split}.
Consider the map
\begin{align}
 G\colon H^1(M)\times H^{\frac{1}{2}}(S)&\to H^{\frac{1}{2},-}(S), \\
 (u,\psi)&\mapsto P^-\left[(1+|\D|)^{-1}(\D\psi-\rho e^{u}\psi)\right]. 
\end{align}
Some explanations are in order. 
Recall that~$H^{\frac{1}{2},-}$ is a Hilbert space, with inner product
\begin{align}
 \left<\psi,\varphi\right>_{_{H^{1/2}}}
 =&\left<\psi,\varphi\right>_{L^2}
 +\left<|\D|^{\frac{1}{2}}\psi,|\D|^{\frac{1}{2}}\varphi\right>_{L^2}\\
 =&\left<(1+|\D|)\psi,\varphi\right>_{H^{-\frac{1}{2}}\times H^{\frac{1}{2}}}.
\end{align}
Now, let ~$G(u,\psi)$ be the element in~$H^{\frac{1}{2}}(S)$ such that, for any~$\varphi\in H^{\frac{1}{2}}(S)$, 
\begin{equation}
 \left<G(u,\psi),\varphi \right>_{_{H^{1/2}}}
 =\left<\D\psi-\rho e^u\psi,P^- (\varphi)\right>_{H^{-\frac{1}{2}}\times H^{\frac{1}{2}}}.
\end{equation}
It follows that~$G(u,\psi) \in H^{\frac{1}{2},-}(S)$, and it is given by the Riesz representation theorem as 
\begin{equation}
 G(u,\psi)=P^-\left[(1+|\D|)^{-1}(\D\psi-\rho e^u\psi)\right].    
\end{equation}

Define the {\em Nehari manifold} ~$N=G^{-1}(0)$, which is non-empty since~$(u,0)\in N$ for any~$u\in H^1(M)$. 
Note that, for each~$u$ fixed, the subset 
\begin{equation}
 N_u\coloneqq\left\{\psi\in H^{\frac{1}{2}}(S)\mid (u,\psi)\in N\right\}=\ker \left[ P^-\circ(1+|\D|)^{-1}\circ (\D- \rho e^u) \right]
\end{equation}
is a  linear subspace (of infinite dimension).
Hence we have a fibration~$N\to H^1(M)$  with fiber~$N_u$ over~$u\in H^1(M)$.
The total space~$N$ is contractible. 
\begin{lemma}
 The Nehari manifold $N$ is a natural constraint  for~$J_\rho$, namely every critical point of $J_\rho|_N$ is 
 an unconstrained critical point of $J_\rho$.  
\end{lemma}
\begin{proof}
 To see that $N$ is a manifold, we show for any~$(u,\psi)$ the surjectivity of the differential~$\dd G(u,\psi)$, which is given by
 \begin{equation}
  \dd G(u,\psi)[v,\phi]=P^-\left[(1+|\D|)^{-1}(\D\phi-\rho e^{u}\phi-\rho e^u v\psi)\right]. 
 \end{equation}
 Restricting to those vectors with~$v=0$ and~$\phi\in H^{\frac{1}{2},-}(S)$, we have 
 \begin{align}
  \left<\dd G(u,\psi)[0,\phi],\phi\right>_{H^{1/2}}
  =&\left<(1+|\D|)^{-1}(\D\phi-\rho e^u\phi),\phi\right>_{H^{1/2}}\\
  =&\int_M \left<\D\phi-\rho e^u\phi,\phi\right>\dv_g \\
  =&-\||\D|^{\frac{1}{2}}\phi\|_{L^2}^2-\rho\int_M e^u|\phi|^2\dv_g. 
 \end{align} 
 Thus~$\left<\dd G(u,\psi)[0,\phi],\phi\right>_{H^{1/2}}$ yields a negative-definite quadratic form on~$H^{\frac{1}{2},-}(S)$.
 In particular,~$\dd G(u,\psi)$ is surjective onto~$H^{\frac{1}{2},-}(S)$, for any~$(u,\psi)$. 
 It follows from the regular value theorem (for an infinite dimensional version, see e.g. \cite{glockner2016fundamentals}) that~$N=G^{-1}(0)$ is a submanifold of~$H^1(M)\times H^{\frac{1}{2}}(S)$.  
 
 Next, we need to show that {if~$(u_0,\psi_0)$ is a critical point of~$J_\rho|_N$, then it is also a critical points of~$J_\rho$ on the full space~$H^1(M)\times H^{\frac{1}{2}}(S)$. }

 Recall that the orthonormal basis~$(\varphi_k)$ for~$H^{\frac{1}{2}}(S)$ consists of eigenspinors. 
 Note that
 \begin{align}
  (u,\psi)\in N
  &\Leftrightarrow  G(u,\psi)=0 
  \Leftrightarrow \int_M \left<\D\psi-\rho e^u \psi, h\right>\dv_g=0, \quad \forall h \in H^{\frac{1}{2},-}\\
  &\Leftrightarrow G_j(u,\psi)\coloneqq \int_M \left<\D\psi-\rho e^u\psi,\varphi_j\right>\dv_g=0, \quad \forall j<0, 
 \end{align}
 that is
 \begin{equation}
  N= G^{-1}(0)=\bigcap_{j<0} G_j^{-1}(0). 
 \end{equation}
 Now let~$(u_0,\psi_0)$ be a critical point of~$J_\rho|_N$: ~$\nabla^N J(u_0,\psi_0)=0$.  
 Then there exist~$\mu_j\in\R$ such that\footnote{To see that such an infinite dimensional version of the Lagrange multiplier theory works, we note that
 \begin{equation}
  \nabla^N J(u_0,\psi_0)= \nabla J(u_0,\psi_0)-\left(\nabla J(u_0,\psi_0)\right)^{\bot} 
 \end{equation}
 where~$\nabla J(u_0,\psi_0)$ denotes the unconstrained gradient and~$(\nabla J(u_0,\psi_0))^\bot$ denotes its normal component.
 Since the gradients~$\{\nabla G_j(u_0,\psi_0): j<0\}$ span the normal space, we can express~$(\nabla J(u_0,\psi_0))^\bot$ in terms of them:
 \begin{equation}
  (\nabla J(u_0,\psi_0))^\bot=\sum_{j<0} \mu_j\nabla G_j(u_0,\psi_0) \in H^{\frac{1}{2}}(S)
 \end{equation}
 for some~$\mu_j\in\R$, ~$j<0$. 
 }
 \begin{equation}\label{eq:Lagrange multiplier}
  \dd J_\rho(u_0,\psi_0)=\sum_{j<0} \mu_j \dd G_j(u_0,\psi_0).
 \end{equation}
 Testing both sides with tangent vectors of the form~$(0,h)$, we have
 \begin{align}
  \int_M\left<\D\psi_0-\rho e^{u_0}\psi_0,h\right>\dv_g
  =\sum_{j<0}\mu_j\int_M\left<\D h-\rho e^{u_0}h, \varphi_j\right>\dv_g.
 \end{align}
 In particular, take~$h = \varphi=\sum_{j<0}\mu_j\varphi_j\in H^{\frac{1}{2},-}$ to obtain
 \begin{align}
  0=\int_M\left<\D\psi_0-\rho e^{u_0}\psi_0,\varphi\right>
  =\int_M\left<\D\varphi-\rho e^{u_0}\varphi,\varphi\right>\dv_g
  \le -C\|\varphi\|^2-\int_M \rho e^{u_0}|\varphi|^2\dv_g.
 \end{align}
 Thus~$\varphi=0$, i.e. ~$\mu_j=0$ for all~$j<0$.  
 Hence in~\eqref{eq:Lagrange multiplier} we have~$\dd J_\rho(u_0,\psi_0)=0$. 
\end{proof}

\

\subsection{Verification of the Palais-Smale condition}
This subsection is devoted to verifying the~$(PS)$ condition for the constrained functional~$J_\rho|_{N}$. 
Note that
\begin{align}
 \dd J_\rho(u,\psi)[v,\phi]
 =\int_M 2(-\Delta u-1+e^{2u}-\rho e^u|\psi|^2)v + 4\left<\D\psi-\rho e^{u}\psi,\phi\right>\dv_g
\end{align}
and for each~$j<0$, with~$G_j$ defined as in the above proof: 
\begin{align}
 \dd G_j(u,\psi)[v,\phi]
 =\int_M \left<\D\phi-\rho e^u\phi,\varphi_j\right>\dv_g -\int_M \rho e^u v\left<\psi,\varphi_j\right>\dv_g.
\end{align}
For each~$(u,\psi)\in N$, there exist constants~$\mu_j(u,\psi)$ such that 
\begin{equation}
 \dd^N J_\rho(u,\psi)
 =\dd J_\rho(u,\psi)-\sum_{j<0} \mu_j(u,\psi)\dd G_j(u,\psi), 
\end{equation}
that is such that for any~$(v,\phi)\in H^1(M)\times H^{\frac{1}{2}}(S)$
\begin{align}
 \dd^N J_\rho(u,\psi)[v,\phi]
 =\dd J(u,\psi)[v,\phi]-\sum_{j<0} \mu_j(u,\psi)\dd G_j(u,\psi)[v,\phi].
\end{align}
Formally  writing~$\varphi(u,\psi)\coloneqq \sum_{j<0} \mu_j\varphi_j$, then 
\begin{align}
\dd^N J(u,\psi)[v,\phi]
=&\int_M 2(-\Delta u-1+e^{2u}-\rho e^u|\psi|^2+\rho e^u\left<\psi, \varphi(u,\psi)\right>)v\dv_g  \\
 &+ \int_M 4 \left( \left<\D\psi-\rho e^{u}\psi,\phi\right>-\left<\D\varphi-\rho e^u\varphi,\phi\right> \right) \dv_g.
\end{align}
Note that this holds for arbitrary~$(v,\phi)$, not only those tangent vectors to~$N$.

Now let~$(u_n,\psi_n)\in N$ be a~$(PS)_c$ sequence for~$J_\rho|_{N}$: this will satisfy 
\begin{equation}\label{eq:PS:at level c}
J_\rho(u_n,\psi_n)=\int_M \left[ |\nabla u_n|^2-2u_n+e^{2u_n}+2\left(\left<\D\psi_n,\psi_n\right>-\rho e^{u_n}|\psi_n|^2\right) \right] \dv_g\to c, 
\end{equation}
\begin{equation}\label{eq:PS:natural constraint}
P^-\circ(1+|\D|)^{-1}\circ (\D\psi_n-\rho e^{u_n}\psi_n)=0; 
\end{equation}
moreover, since the differential of $J_\rho$  is tending to zero only when applied 
to vectors tangent  to $N$, there exists 
 some~$\varphi_n\in H^{\frac{1}{2},-}(S)$ 
such that 
\begin{equation}\label{eq:PS:function part}
2(-\Delta u_n-1+e^{2u_n}-\rho e^{u_n}|\psi_n|^2)-\rho e^{u_n}\left<\psi_n,\varphi_n\right>=\alpha_n\to 0 \textnormal{ in \;} H^{-1}(M),
\end{equation}
\begin{equation}\label{eq:PS:spinor part}
4(\D\psi_n-\rho e^{u_n}\psi_n)-(\D\varphi_n-\rho e^{u_n}\varphi_n)=\beta_n \to 0 \textnormal{ in } H^{-\frac{1}{2}}(S).
\end{equation}

\medspace 

\begin{lemma}
With the same notation as above, we have 
\begin{enumerate}
 \item The auxiliary spinors~$\varphi_n$ satisfy~$\|\varphi_n\|_{H^{\frac{1}{2}}}\to 0$ as~$n\to\infty$. 
 \item The sequence~$(u_n,\psi_n)$ is uniformly bounded (with bounds depending on the level~$c$) in~$H^1(M)\times H^{\frac{1}{2}}(S)$.
\end{enumerate} 
\end{lemma}

\begin{proof}
 (1) Testing~\eqref{eq:PS:natural constraint} against~$\varphi_n$ we find 
     \begin{equation}
     \int_M \left<\D\psi_n-\rho e^{u_n}\psi_n,\varphi_n\right>\dv_g=0, 
     \end{equation}
     while testing ~\eqref{eq:PS:spinor part} against~$\varphi_n$ we get 
	\begin{equation}
	 -\int_M \left<\D\varphi_n,\varphi_n\right>\dv_g +\rho\int_M e^{u_n}|\varphi_n|^2\dv_g=\left<\beta_n,\varphi_n\right>.
	\end{equation}	     
    Since~$\varphi_n$ lies in the span of the negative eigenspinors, we see that 
    \begin{equation}
    C\|\varphi_n\|_{H^{\frac{1}{2}}}^2+\rho\int_M e^{u_n}|\varphi_n|^2\dv_g=o(\|\varphi_n\|_{H^{\frac{1}{2}}}). 
    \end{equation}
    It follows that as~$n\to\infty$, 
    \begin{align}
    \|\varphi_n\|_{H^{\frac{1}{2}}}\to 0,  \qquad \int_M \rho e^{u_n}|\varphi_n|^2\dv_g \to 0. 
    \end{align}    
    
\

(2) Testing~\eqref{eq:PS:function part} against~$v\equiv 1\in H^1(M)$, we obtain 
	\begin{align}
	2\int_M e^{2u_n}\dv_g -2\int_M \dv_g-2\rho\int_M \left( e^{u_n}|\psi_n|^2-e^{u_n}\left<\psi_n,\varphi_n\right> \right)\dv_g
	=\left<\alpha_n,1\right>_{H^{-1}\times H^1},
	\end{align}
	which can be read as 
	\begin{equation}\label{eq:estimate of e-2u}
	\int_M e^{2u_n}\dv_g
	=4\pi(\genus-1)+\rho\int_M e^{u_n}|\psi_n|^2\dv_g+\frac{1}{2}\rho\int_M e^{u_n}\left<\psi_n,\varphi_n\right>\dv_g +o(1).
	\end{equation}
	Now we can control the second integral on the right-hand side by
	\begin{equation}\label{eq:estimate of mixed term:spinor-multiplier-fucntion}
	\left|\frac{\rho}{2}\int_M e^{u_n}\left<\psi_n,\varphi_n\right>\dv_g\right|
	\le \varepsilon\int_M \rho e^{u_n}|\psi_n|^2\dv_g+\varepsilon\int_M e^{2u_n}\dv_g+ C(\varepsilon,\rho)\|\varphi_n\|^4, 
	\end{equation}
    where~$\varepsilon>0$ is some small number. 
    Substituting this into~\eqref{eq:estimate of e-2u} and noting that~$\|\varphi_n\|=o(1)$, we get
    \begin{equation}
    \int_M e^{2u_n}\ge \frac{4\pi(\genus-1)}{1+\varepsilon}+\frac{1-\varepsilon}{1+\varepsilon}\int_M \rho e^{u_n}|\psi_n|^2\dv_g+o(1).
    \end{equation}
    
	Testing ~\eqref{eq:PS:spinor part} against~$\psi_n$ we deduce 
	\begin{equation}
	4\int_M \left( \left<\D\psi_n,\psi_n\right>-\rho e^{u_n}|\psi_n|^2 \right)\dv_g -\int_M \left<\D\varphi_n-\rho e^{u_n}\varphi_n,\psi_n\right>\dv_g
	=\left<\beta_n,\psi_n\right>_{H^{-\frac{1}{2}}\times H^{\frac{1}{2}}}.
	\end{equation}
	Since the second integral vanishes because of~\eqref{eq:PS:natural constraint}, we thus get
	\begin{equation}
	\int_M \left( \left<\D\psi_n,\psi_n\right>-\rho e^{u_n}|\psi_n|^2 \right) \dv_g =o(\|\psi_n\|).
	\end{equation}
	Combining these  estimates with~\eqref{eq:PS:at level c} we see that
	\begin{align}
	c+o(1)
	=&\int_M|\nabla\widehat{u}_n|^2\dv_g-8\pi(\genus-1)\bar{u}_n+\int_M e^{2u_n}\dv_g
	+2\int_M \left(\left<\D\psi_n,\psi_n\right>-\rho e^{u_n}|\psi_n|^2 \right)\dv_g  \\
	\ge&\int_M|\nabla\widehat{u}_n|^2\dv_g-4\pi(\genus-1)(2\bar{u}_n-\frac{1}{1+\varepsilon})
	+\frac{1-\varepsilon}{1+\varepsilon}\rho\int_M e^{u_n}|\psi_n|^2\dv_g
	 +C(\varepsilon,\rho)o(1)+o(\|\psi_n\|),
	\end{align}
	which is to say,
	\begin{align}\label{eq:estimate of Dirichlet and mix-spinor-function}
	\int_M |\nabla \widehat{u}_n|^2+\frac{1-\varepsilon}{1+\varepsilon}\rho e^{u_n}|\psi_n|^2\dv_g
	\le c+4\pi(\genus-1)(2\bar{u}_n-\frac{1}{1+\varepsilon})+C(\varepsilon,\rho)o(1)+o(\|\psi_n\|).
	\end{align}
	
	\medskip
	
	Now we estimate the averages~$\bar{u}_n$. 
	Note that by~\eqref{eq:estimate of e-2u} and~\eqref{eq:estimate of mixed term:spinor-multiplier-fucntion} we also obtain
	\begin{align}
	 \int_M e^{2u_n}\dv_g 
	 \le \frac{4\pi(\genus-1)}{1-\varepsilon}
			+\frac{1+\varepsilon}{1-\varepsilon}\int_M \rho e^{u_n}|\psi_n|^2\dv_g +C(\varepsilon,\rho)o(1).	 
	\end{align}
	Then by Jensen's inequality,
	\begin{align}
	e^{2\bar{u}_n}
	\le & e^{2\bar{u}_n}\fint_M e^{2\widehat{u}_n}\dv_g =\frac{1}{4\pi(\genus-1)}\int_M e^{2u_n}\dv_g \\
	\le& \frac{1}{1-\varepsilon}+\frac{1}{4\pi(\genus-1)}\frac{1+\varepsilon}{1-\varepsilon}\int_M \rho e^{u_n}|\psi_n|^2\dv_g+ C(\varepsilon,\rho,\genus) o(1) \\
	\le&\frac{1}{1-\varepsilon}+\left(\frac{1+\varepsilon}{1-\varepsilon}\right)^2 \left(\frac{c}{4\pi(\genus-1)}+2\bar{u}_n-\frac{1}{1+\varepsilon}\right)+ C(\varepsilon,\rho,\genus) o(1)+ C(\varepsilon,\genus) o(\|\psi_n\|). 
	\end{align}
	Thus there exists ~$C=C(\varepsilon,\rho,\genus)>0$ such that 
	\begin{equation}
	|\bar{u}_n|\le C\bigr(1+c+o(\|\psi_n\|)\bigr).
	\end{equation}
	
	\medskip
	
	The spinors can be controlled by the above growth estimates. 
	Testing~\eqref{eq:PS:spinor part} against~$\psi_n^+$, we find 
	\begin{align}
	4\int_M \left( \left<\D\psi_n,\psi_n^+\right>-\rho e^{u_n}\left<\psi_n,\psi_n^+\right> \right) \dv_g
	-\int_M \left<\D\varphi_n-\rho e^{u_n}\varphi_n,\psi_n^+\right>\dv_g
	=\left<\beta_n,\psi_n^+\right>_{H^{-\frac{1}{2}}\times H^{\frac{1}{2}}}.
	\end{align}
	It follows that
	\begin{align}
	C\|\psi_n^+\|^2
	\le& \int_M\left<\D\psi_n,\psi^+\right>\dv_g
	    =\int_M\rho e^{u_n}\left<\psi_n,\psi_n^+\right>\dv_g
	    +\frac{1}{4}\int_M\left<\D\varphi_n-\rho e^{u_n}\varphi_n,\psi_n^+\right>\dv_g+o(\|\psi_n^+\|)\\
	\le& \left(\int_M e^{u_n}|\psi_n|^2\dv_g\right)^{\frac{1}{2}}
			\left(\int_M e^{2u_n}\dv_g\right)^{\frac{1}{4}}\left(\int_M|\psi_n^+|^4\dv_g\right)\\
		&	+\|\varphi_n\|\|\psi_n^+\|
			+\rho\left(\int_M e^{2u_n}\dv_g\right)^{\frac{1}{2}}\|\varphi_n\| \|\psi_n^+\| +o(\|\psi_n^+\|) \\
    \le& C\left(1+c+o(\|\psi_n\|)\right)^{\frac{3}{4}}\|\psi_n^+\| + o(\|\psi_n^+\|).
	\end{align}
 	For what concerns  the other component~$\psi_n^-$, we use~\eqref{eq:PS:natural constraint} to get
 	\begin{align}
 	C\|\psi_n^-\|^2
 	\le &-\int_M \left<\D\psi_n^-,\psi_n\right>\dv_g
 		=-\rho\int_M e^{u_n}\left<\psi_n,\psi_n^-\right>\dv_g \\
 	\le& \rho\left(\int_M e^{2u_n}\dv_g\right)^{\frac{1}{4}} \left(\int_M e^{u_n}|\psi_n|^2\dv_g\right)^{\frac{1}{2}}\|\psi_n^-\|\\
 	\le& C(1+c+o(\|\psi_n\|))^{\frac{3}{4}} \|\psi_n^-\|.
 	\end{align}
 	Consequently,
 	\begin{equation}
 	\|\psi_n\|^2=\|\psi_n^+\|^2+\|\psi_n^-\|^2 
 	\le C\bigr(1+c+o(\|\psi_n\|)\bigr)^{\frac{3}{4}}\|\psi_n\| + o(\|\psi_n\|).
 	\end{equation}
     Thus there exists some constant~$C=C(c,\genus,\rho)>0$ such that
 	\begin{equation}
  	\|\psi_n\|\le C(c,\genus,\rho)<+\infty.
 	\end{equation}
 	This uniform bound (depending on the level~$c$) in turn gives bounds on ~$\bar{u}_n$ and thus
 	\begin{equation}
  	\int_M \left( |\nabla \widehat{u}_n|^2+\rho e^{u_n}|\psi_n|^2 \right) \dv_g 
  	\le C'(c,\genus,\rho)<\infty. 
 	\end{equation}
\end{proof}

Now, passing to a subsequence if necessary, we may assume that there exist~$u_\infty\in H^1(M)$ and~$\psi_\infty\in H^{\frac{1}{2}}(S)$ such that
\begin{align}
 &u_n\rightharpoonup u_\infty \textnormal{ weakly in } H^1(M), \\
 &\psi_n\rightharpoonup \psi_\infty \textnormal{ weakly in } H^{\frac{1}{2}}(S). 
\end{align}

\medspace 

\begin{lemma}
 The pair~$(u_\infty,\psi_\infty)$ is a smooth solution of~\eqref{eq:EL-uniformized metric}. 
\end{lemma}
\begin{proof}
 According to the compactness of the Moser--Trudinger embedding (\cite[theorem 2.46]{aubin1998somenonlinear}), we see that
 \begin{equation}
  e^{u_n} \to e^{u_\infty} \textnormal{ strongly in } L^p(M),\quad   (p<\infty).
 \end{equation}
 Meanwhile, thanks to Rellich-Kondrachov compact embedding Theorem (see e.g.~\cite{gilbarg2001elliptic})
 \begin{equation}
  \psi_n\to \psi_\infty \textnormal{ strongly in }  L^q(S), \quad(q<4).
 \end{equation}
 Hence~$e^{u_n}|\psi_n|^2$ converges weakly in~$L^p$ to~$e^{u_\infty}|\psi_\infty|^2$, for any~$p<2$.

 It follows that $(u_\infty, \psi_\infty)$ is a weak solution to~\eqref{eq:EL-uniformized metric}.
 As remarked before, any weak solution is a classical, hence smooth, solution.  
\end{proof}

In particular, this implies that the weak limit~$(u_\infty,\psi_\infty)$ is in the Nehari manifold~$N$.  
Consider the differences
\begin{align}
 &v_n\coloneqq u_n-u_\infty, \\
 &\phi_n\coloneqq \psi_n-\psi_\infty.
\end{align}
Then~$(v_n,\phi_n)\rightharpoonup (0,0)$ weakly in~$H^1(M)\times H^{\frac{1}{2}}(S)$. 
The functions~$v_n$ satisfy
\begin{align}
 \Delta_g v_n
 =(e^{2u_n}-e^{2u_\infty})-\rho\left(e^{u_n}|\psi_n|^2-e^{u_\infty}|\psi_\infty|^2\right), 
\end{align}
where the right-hand sides converge to~$0$ in~$L^q(M)$ for any~$q<2$.  
This implies that the~$v_n$'s are uniformly bounded in~$H^1(M)$ and converge strongly to a limit function~$v_\infty\in H^1(M)$ satisfying
\begin{equation}
 \Delta_g v_\infty=0.
\end{equation}
Thus~$v_\infty=const.$, which has to be zero since~$v_n\rightharpoonup 0$. 
This implies the strong convergence of~$u_n$ to~$u_\infty$ in~$H^1(M)$. 
Let us look next at the equations for~$\phi_n$'s:
\begin{equation}
 \D_g\phi_n=\rho e^{u_n}\psi_n- \rho e^{u_\infty}\psi_\infty, 
\end{equation}
where the right-hand sides converges to~$0$ in~$L^q(S)$ for any~$q<4$.
Thus there exists ~$\phi_\infty\in H^{\frac{1}{2}}(S)$ such that~$\phi_n$ converges strongly in~$H^{\frac{1}{2}}(S)$ to~$\phi_\infty$, which satisfies
\begin{equation}
 \D_g\phi_\infty=0.
\end{equation}
By the assumption of trivial kernel on $\D_g$, we have that~$\phi_\infty=0$, that is ~$\psi_n$ converges strongly to~$\psi_\infty$ in~$H^{\frac{1}{2}}(S)$. 
Since~$N$ is a submanifold of~$H^1(M)\times H^{\frac{1}{2}}(S)$, the sequence~$(u_n,\psi_n)$ also converges inside~$N$ to~$(u_\infty,\psi_\infty)$; in other words,~$(u_\infty,\psi_\infty)$ lies in the closure of~$\{(u_n,\psi_n)\}$ relative to~$N$. 
Thus we verified the following 

\

\begin{prop}
 The functional~$J_\rho|_N$ satisfies the Palais-Smale condition. 
\end{prop}

\

%%%%%%%%%%%%%%%%%%%
\section{Mountain pass and linking geometry}
In this section we will show that the functional~$J_\rho|_N$, for suitable~$\rho$'s, possesses either a mountain pass or linking geometry around the trivial solution~$(0,0)$, which will yield existence of a non-trivial min-max critical point.  

For later convenience let us introduce the notation
\begin{align}
 F(u)\coloneqq\int_M \left( |\nabla u|^2-2u+e^{2u} \right) \dv_g, & & 
 Q(u,\psi)\coloneqq2\int_M \left( \left<\D\psi,\psi\right>-\rho e^{u}|\psi|^2 \right) \dv_g.
\end{align}
Then we have
\begin{enumerate}
 \item[(i)] $J_\rho(u,\psi)=F(u)+Q(u,\psi)$.
 \item[(ii)] $F(u) \geq 4\pi(\genus-1)$, and this lower bound is achieved by the unique minimizer~$u_{min}\equiv0$.  
 \item[(iii)] $Q(u,\psi)$ is quadratic in~$\psi$ and  strongly indefinite. 
\end{enumerate}

\

\subsection{Local behavior near (0,0)}

Let~$(u,\psi)\in N$ be \emph{close} to~$(0,0)$. 
The constraint that defines $N$, i.e.~$P^-(1+|\D|)^{-1}(\D\psi-\rho e^u\psi)=0$, implies
\begin{equation}
 \int_M\left<\D\psi-\rho e^u\psi,P^-\psi\right>\dv_g=0.
\end{equation}
Hence we get
\begin{align}
 -\int_M\left<\D\psi,\psi^-\right>\dv_g
 =&-\rho\int_M e^{u}\left<\psi^++\psi^-,\psi^-\right> \\
 =&-\rho\int_M e^{u}|\psi^-|^2\dv_g 
   -\rho\int_M e^{u}\left<\psi^+,\psi^-\right>\dv_g.
\end{align}
Since~$\|e^u\|_{L^p}\le C(1+\|u\|_{H^1})\le C$ for~$\|u\|$ uniformly  bounded, 
we have
\begin{equation}\label{eq:positive spinor dominates negative part}
 \|\psi^-\|\le C\rho\|\psi^+\|.
\end{equation}

Now consider the functional
\begin{align}\label{eq:local behavior of J rho}
 J_\rho(u,\psi)
 & = F(u)+Q(u,\psi)=F(u)+2\int_M \left( \left<\D\psi,\psi\right>-\rho e^u|\psi|^2 \right) \dv_g \\
 & = F(u)+2\int_M\left<(\D-\rho)\psi, \psi^+\right>\dv_g
 +2\int_M \rho (1-e^u)\left<\psi,\psi^+\right>\dv_g.
\end{align}
The last integral is now of cubic order in $(u,\psi)$, i.e.:
\begin{align}
 2\int_M \rho (1-e^u)\left<\psi,\psi^+\right>\dv_g
 \le C\|u\|_{H^1} \|\psi\|\|\psi^+\|.
\end{align}
For the first term, if we take the equivalent norm~$|\bar{u}|^2+\|\nabla \widehat{u}\|_{L^2}^2\sim \|u\|_{H^1}^2$, then for~$t^2=|\bar{u}|^2+\|\nabla \widehat{u}\|_{L^2}^2>0$
\begin{itemize}
 \item if~$|\bar{u}|^2\ge\frac{t^2}{2}\ge \|\nabla\widehat{u}\|_{L^2}^2$, then 
       \begin{equation}
        F(u)\ge \int_M \left( e^{2\bar{u}}-2\bar{u} \right) \dv_g \ge 4\pi(\genus-1)+Ct^2,
       \end{equation}
 \item if~$\|\nabla\widehat{u}\|_{L^2}^2\ge\frac{t^2}{2}\ge|\bar{u}|^2$, then 
      \begin{equation}
       F(u)\ge \int_M \left( |\nabla\widehat{u}|^2+1 \right) \dv_g\ge4\pi(\genus-1)+\frac{1}{2}t^2, 
      \end{equation}
\end{itemize}
thus in either case we have
\begin{equation}
 F(u)\ge 4\pi(\genus-1)+C^{-1}\|u\|_{H^1}^2. 
\end{equation}
It remains to analyze the middle integral term in the r.h.s. of~\eqref{eq:local behavior of J rho}. 
As before, we write~$\psi=\sum_{j\in \mathbb{Z}_*} a_j\varphi_j$: 
then
\begin{align}
 2\int_M\left<(\D-\rho)\psi, \psi^+\right>\dv_g
 =\sum_{j>0} 2(\lambda_j-\rho)a_j^2.
\end{align}

From now on we assume that~$\rho\notin\Spec(\D)$. Thus the above summation can be split into two parts
\begin{equation}
 2\int_M\left<(\D-\rho)\psi, \psi^+\right>\dv_g
 =-\sum_{0<\lambda_j<\rho} 2(\rho-\lambda_j)a_j^2
 +\sum_{\lambda_j>\rho} 2(\lambda_j-\rho)a_j^2. 
\end{equation}

\subsubsection{Mountain pass geometry}
First we consider the easier case~$0<\rho<\lambda_1$, so the first part of the above summation vanishes. 
Then, locally near~$(0,0)$ in ~$N$, we have
\begin{align}
 J_\rho(u,\psi) 
 \ge& 4\pi(\genus-1)+C^{-1}\|u\|_{H^1}^2+ C^{-1}\left(1-\frac{\rho}{\lambda_1}\right)\|\psi^+\|^2-C\|u\|_{H^1}\|\psi\|\|\psi^+\| \\
 \ge& 4\pi(\genus-1)+C^{-1}\|u\|_{H^1}^2+ C^{-1}\left(1-\frac{\rho}{\lambda_1}-C^2\|\psi\|^2\right)\|\psi^+\|^2, 
\end{align}
where we have used Cauchy-Schwarz inequality for the last   term, of cubic order. 
It follows that when
$\|u\|_{H^1}^2+\|\psi\|_{H^{\frac{1}{2}}}^2=r^2>0$ is small, there exists a continuous function~$\theta(r)>0$ such that
\begin{equation}\label{eq:l-bd}
 J_\rho(u,\psi)\ge J(0,0)+\theta(r). 
\end{equation}
On the other hand, we can choose a large constant~$\bar{u}_1\in H^1(M)$ such that~$\rho e^{\bar{u}_1}>\lambda_1+1$ and then take~$s>0$ large such that
\begin{align}
 J(\bar{u}_1, s\varphi_1) 
 =&vol(M,g)(e^{2\bar{u}_1}-2\bar{u}_1) 
  +2(\lambda_1-\rho e^{\bar{u}_1})s^2 \\
 =&4\pi(\genus-1)(e^{2\bar{u}_1}-2\bar{u}_1)
  -2(\rho e^{\bar{u}_1}-\lambda_1)s^2
\end{align}
is negative. 
Thus we have the mountain pass geometry locally near ~$(0,0)$ in the Nehari manifold~$N$. 
Let~$\Gamma$ be the space of paths connecting~$(0,0)$ and~$(\bar{u}_1, s\varphi_1)$ inside~$N$ (notice that 
$\Gamma \neq \emptyset$ since $N$ is contractible, and hence connected), parametrized by~$t\in [0,1]$, and define
\begin{equation}
 c_1\coloneqq \inf_{\alpha\in\Gamma} \sup_{t\in [0,1]} J_\rho(\alpha(t)).
\end{equation}
From the above arguments we have that~$c_1>4\pi(\genus-1)$.
It follows that~$c_1$ is a critical value for~$J_\rho$ with a critical point at this level, which is different from the trivial one. 
This concludes the proof of Theorem~\ref{thm} in this case. 

\medskip

\subsubsection{Linking geometry}
Next we consider the case~$\rho\in(\lambda_k,\lambda_{k+1})$ for some~$k\ge 1$. 
Now there are more directions in which~$J_\rho$ becomes negative, but these are at most finitely-many, and we will apply a linking method to 
exploit the geometry of the functional. 

 Decomposing first the space~$H^{\frac{1}{2}}(S)$ into two parts: 
\begin{align}
 H^{\frac{1}{2},k+}\coloneqq &
 \left\{\phi_1\in H^{\frac{1}{2}}(S)\mid \phi_1=\sum_{j>k}a_j\varphi_j\right\}, \\
 H^{\frac{1}{2},k-}\coloneqq &
 \left\{\phi_2\in H^{\frac{1}{2}}(S)\mid \phi_2=\sum_{j\le k}a_j\varphi_j\right\},
\end{align}
we have then the orthogonal decomposition 
\begin{equation}
 H^{\frac{1}{2}}(S)=H^{\frac{1}{2},k+}\oplus 
 \left(H^{\frac{1}{2},k-}\cap H^{\frac{1}{2},+}(S)\right)
 \oplus H^{\frac{1}{2},-}(S). 
\end{equation}

Now consider the set
\begin{equation}
 \mathscr{N}_k\coloneqq
 \{0\}\times \left(H^{\frac{1}{2},k-}\cap H^{\frac{1}{2},+}(S)\right)\subset H^1(M)\times H^{\frac{1}{2}}(S).
\end{equation}
It is easy to see that~$\mathscr{N}_k$ is a linear subspace inside~$N$, and along this subspace the functional~$J_\rho$ is not larger than the minimal critical value:
\begin{equation}
 J_\rho(0,\phi_1)
 =4\pi(\genus-1)-2\sum_{0<j\le k} (\rho-\lambda_j) a_j^2 
 \le 4\pi(\genus -1).
\end{equation}
For $\tau > 0$ let us consider the following cone around~$\mathscr{N}_k$:
\begin{multline}
 \mathcal{C}_\tau(\mathscr{N}_k)\coloneqq 
 \Big\{(u,\psi)\in N
 \mid u\in H^1(M), \psi=\phi_1+\phi_2+\psi^-\in H^{\frac{1}{2},k+}\oplus (H^{\frac{1}{2},k-}\cap H^{\frac{1}{2},+}(S))\oplus H^{\frac{1}{2},-}(S), \\
 \|u\|_{H^1}+\|\phi_1\|^2+\|\psi^-\|^2< \tau \|\phi_2\|^2 \Big\}. 
\end{multline}  
We claim that for~$\tau$ suitably chosen this cone contains all the decreasing directions, in the sense that outside the cone we can find a region on which the functional is strictly above the ground state level.

Letting~$(u,\psi)\in N\setminus \mathcal{C}_\tau(\mathscr{N}_k)$, i.e., with~$\psi=\phi_1+\phi_2$ decomposed as above, we have
\begin{equation}
 \|u\|^2_{H^1}+\|\phi_1\|^2+\|\psi^-\|^2\ge \tau \|\phi_2\|^2.
\end{equation}
By~\eqref{eq:positive spinor dominates negative part}, which can be now interpreted as
\begin{equation}
 \|\psi^-\|^2\le C\rho^2(\|\phi_1\|^2+\|\phi_2\|^2), 
\end{equation}
we see that 
\begin{equation}\label{eq:lower part of spinor is dominated by higher part}
 \|\phi_2\|^2\le \frac{1}{\tau-C\rho^2}\left(\|u\|^2_{H^1}+(1+C\rho^2)\|\phi_1\|^2\right).
\end{equation}
Moreover, this also implies that 
\begin{equation}
 \|u\|_{H^1}^2+ \|\phi_1\|^2 \ge C(\|u\|^2 +\|\psi\|^2),
\end{equation}
for some~$C=C(\rho,\tau)>0$.

Then in~\eqref{eq:local behavior of J rho} for the scalar component we have as before the control
\begin{equation}
 F(u)\ge 4\pi(\genus-1)+C\|u\|_{H^1}^2.
\end{equation}
For the spinorial part, since~$\psi=\phi_1+\phi_2+\psi^-$ is an orthogonal decomposition, we have
\begin{align}
 Q(u,\psi)
 =&2\int_M\left<(\D-\rho)\psi,\psi^+\right>\dv_g
   +2\rho\int_M (1-e^u)\left<\psi,\psi^+\right>\dv_g \\
 =&2\int_M\left<(\D-\rho)\phi_1,\phi_1\right>\dv_g
 +2\int_M\left<(\D-\rho)\phi_2,\phi_2\right>\dv_g \\
 &+2\rho\int_M (1-e^u)\left<\psi,\phi_1+\phi_2\right>\dv_g \\
 \ge&C\left(1-\frac{\rho}{\lambda_{k+1}}\right)\|\phi_1\|^2
     -C\left(\frac{\rho}{\lambda_k}-1\right)\|\phi_2\|^2 
     -C\rho\|u\|_{H^1}\|\psi\|(\|\phi_1\|+\|\phi_2\|).
\end{align}
Assuming~$\|u\|^2_{H^1}+\|\psi\|^2=r^2$ is small and noting~\eqref{eq:lower part of spinor is dominated by higher part}, we get
\begin{align}
 J_\rho(u,\phi_1+\phi_2)
 \ge& 4\pi(\genus-1)+C\|u\|_{H^1}^2
 + C\left(1-\frac{\rho}{\lambda_{k+1}}-Cr^2\right)\|\phi_1\|^2\\
 & -C\left(\frac{\rho}{\lambda_k}-1-Cr^2\right)\|\phi_2\|^2 \\
 \ge&4\pi(\genus-1)+C\|u\|_{H^1}^2
 + C\left(1-\frac{\rho}{\lambda_{k+1}}-Cr^2\right)\|\phi_1\|^2\\
 & -C\left(\frac{\rho}{\lambda_k}-1\right)\frac{\|u\|_{H^1}^2+(1+C\rho^2)\|\phi_1\|^2}{\tau-C\rho^2}.
\end{align}
We can first choose~$r$ small enough and then choose~$\tau$ large enough such that
\begin{align}
 J_\rho(u,\phi_1+\phi_2)
 \ge& 4\pi(\genus-1)+ C(\|u\|^2+\|\phi_1\|^2) \\
 \ge& 4\pi(\genus-1)+C r^2
\end{align}
outside the cone~$\mathcal{C}(\mathscr{N}_k)$. 
Thus the claim is confirmed. 

\medskip

For~$r$ as above, consider the set
\begin{equation}
 L_1\coloneqq\bigr( \p B_r(0,0)\backslash \mathcal{C}(\mathscr{N}_k)\bigr) \cap N, 
\end{equation}
which is non-empty since~$(0,r\varphi_{k+1})\in L_1$.
Recall that~$N$ is locally modeled by a Hilbert space, e.g. $T_{(0,0)} N$. 
We can assume that in a local chart,~$\mathscr{N}_k$ is some {\em coordinate subspace}, while~$L_1$ is homeomorphic to a collar neighborhood of the sphere (of infinite dimension) which lies in a subspace complementary to~$\mathscr{N}_k$ and intersects $\mathcal{C}_\tau(\mathscr{N}_k)$ only at $\{(0,0)\}$. 

Next we introduce a set~$L_2$ on which the functional attains low values and such that it links with~$L_1$, see Figure~\ref{fig:linking}. 
The construction of such a set is performed in several steps. First we take the ball
\begin{equation}
 B_R^{0,k}(0)\coloneqq\Bigr\{(0,\phi_2)\in\mathscr{N}_k \mid \|\phi_2\|\le R\Bigr\}\subset \mathscr{N}_k
\end{equation}
with~$R>0$ a large constant to be fixed later. 
Note that for any~$(0,\phi_2)$,~$J_\rho(0,\phi_2)\le 4\pi(\genus-1)$ and for~$(0,\phi_2)\in \p B_R^{0,k}(0)$, 
\begin{equation}
 J_\rho(0,\phi_2)
 \le 4\pi(\genus-1)-C\left(\frac{\rho}{\lambda_k}-1\right)\|\phi_2\|^2
 \le 4\pi(\genus-1)-C\left(\frac{\rho}{\lambda_k}-1\right)R^2.
\end{equation}
For any~$(0,\phi_2)\in\p B_R^{0,k}(0)$, we consider the following curves. 
First let
\begin{equation}
 \sigma_{1}\colon [0,T]\to N, \quad
 \sigma_1(t)\coloneqq(t, \phi_2+At\varphi_{k+1}),
\end{equation}
where~$A>0$ is again a constant to be fixed later.
One easily sees that this is a curve in~$N$ and 
\begin{align}
 J_\rho(t,\phi_2+&At\varphi_{k+1})
 =vol(M,g)(e^{2t}-2t)+2\int_M\left<(\D-\rho e^t)\phi_2,\phi_2\right>\dv_g\\
  &\qquad\qquad+2A^2t^2\int_M\left<(\D-\rho e^t)\varphi_{k+1},\varphi_{k+1}\right>\dv_g\\
 \le&4\pi(\genus-1)(e^{2t}-2t)
 +2\int_M\left<(\D-\rho)\phi_2,\phi_2\right>\dv_g 
 +2A^2t^2 (\lambda_{k+1}-\rho e^t) \\
 \le&4\pi(\genus-1)(e^{2t}-2t)
 -C\left(\frac{\rho}{\lambda_k}-1\right)R^2
 +2A^2t^2 (\lambda_{k+1}-\rho e^t). 
\end{align}
Now we fix some constants:
\begin{itemize}
 \item we choose~$T>0$ such that~$\rho e^T-\lambda_{k+1}\ge 1$;
 \item then we choose~$A>0$ such that
       \begin{equation}
        4\pi(\genus-1)(e^{2T}-2T)-2A^2T^2(\rho e^T-\lambda_{k+1})<4\pi(\genus-1); 
       \end{equation}
  \item finally, choose~$R>0$ such that for any~$t\in [0,T]$
        \begin{equation}
         4\pi(\genus-1)(e^{2t}-2t)
 -C\left(\frac{\rho}{\lambda_k}-1\right)R^2
 +2A^2t^2 (\lambda_{k+1}-\rho e^t)<4\pi(\genus-1).
        \end{equation}
\end{itemize}
Then we consider the curve
\begin{align}
 \sigma_2\colon [-1,1]\to N, \quad
 \sigma_2(r)\coloneqq \bigr(T,(-r)\phi_2+AT\varphi_{k+1}\bigr),
\end{align} 
which joins~$(T,\phi_2+AT\varphi_{k+1})$ to~$(T,-\phi_2+AT\varphi_{k+1})$ inside~$N$. 
Thereafter we can come back to~$\mathscr{N}_k$ via the curve
\begin{equation}
 \sigma_3\colon [0,T]\to N, \quad 
 \sigma_3(t)\coloneqq \bigr((T-t), \phi_2+A(T-t)\varphi_{k+1}\bigr).  
\end{equation}
Finally, consider the subset
\begin{equation}
 \mathcal{D}\coloneqq
 \left\{(t,At\varphi_{k+1}+\phi_2) \mid t\in [0,T], (0, \phi_2)\in B_R^{0,k}(0) \right\}, 
\end{equation} 
which is compact and  homeomorphic to a finite-dimensional cylindrical segment 
\begin{equation}
 [0,T]\times B_R^{0,k}(0).
\end{equation}
Note that~$\mathcal{D}\subset N$ and
let~$L_2=\p \mathcal{D}$, see Figure~\ref{fig:linking}.    
The curves $\sigma_1, \sigma_2, \sigma_3$ constructed above pass through every point of~$L_2\backslash B_R^{0,k}(0)$. It follows that on~$L_2$ the functional attains low values.  
One can shrink~$L_2$ (in an homotopically equivalent way) into the coordinate chart to see that~$L_1$ and~$L_2$ actually link,  see e.g. \cite{ambrosetti2007nonlinear, struwe2008variational} for a rigorous definition of this concept. 

\begin{figure}[h]  
\centering
\includegraphics[width=0.5\linewidth]{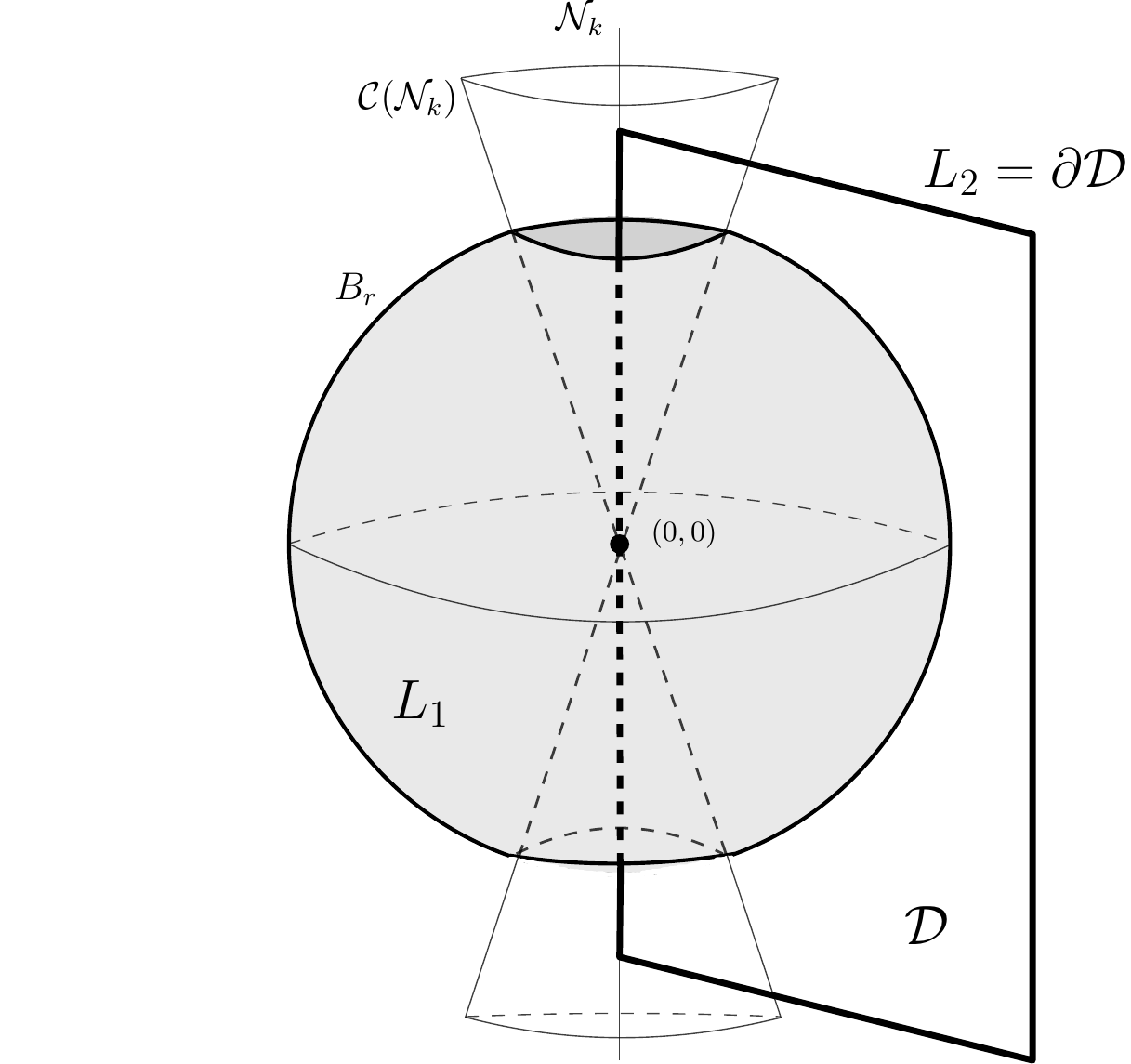}
\caption{}
\label{fig:linking}
\end{figure}

\medskip

Now we define the linking level.
Let~$\Gamma$ be the space of continuous maps
\begin{equation}
 \alpha\colon \mathcal{D}\to N,  
\end{equation}
such that~$\alpha(v,h)=(v,h)$ for any~$(v,h)\in L_2=\p \mathcal{D}$. 
This set~$\Gamma$ is clearly non-empty since~$\Id_\mathcal{D}\in\Gamma$. 
Then we define the linking level
\begin{equation}
 c_1\coloneqq \inf_{\alpha\in\Gamma}\max_{(v,h)\in \mathcal{D}}J_{\rho}(\alpha(v,h)) .
\end{equation}
As~$L_1$ and~$L_2$ link, we have from the above arguments that
\begin{equation}
 c_1 \ge 4\pi(\genus-1)+\theta(r). 
\end{equation}
It follows that~$c_1$ is a critical value for~$J_\rho$, and again we obtain a critical point for~$J_\rho $ which is different from the trivial one.        
This concludes the proof of Theorem~\ref{thm} in this case as well.
  
\


\begin{thebibliography}{HKW}


\bibitem{super1}
Changrim Ahn, Chaiho Rim, Marian Stanishkov. Exact one-point function of N=1 super-Liouville theory with boundary. \emph{Nucl. Phys. B} \textbf{636} (2002), 497--513.

\bibitem{ambrosetti2007nonlinear}
Antonio Ambrosetti, Andrea Malchiodi.
\emph{Nonlinear Analysis and Semilinear Elliptic Problems.} 
Cambridge Studies in Advanced Mathematics, 104. Cambridge University Press, Cambridge, 2007.

\bibitem{ammann2003habilitation}
Bernd Ammann.
\emph{A variational problem in conformal spin geometry.}
Habilitation, Universit\"at Hamburg, 
2003.


\bibitem{aubin1998somenonlinear}
Thierry Aubin.
\emph{Some nonlinear problems in Riemannian geometry.} 
Springer, Berlin,
1998. 

\bibitem{bar1992harmonic}
Christian B\"ar, Paul Schmutz.
Harmonic spinors on Riemann surfaces.
\emph{Annals of Global Analysis and Geometry} \textbf{10} (1992), 263--273.

\bibitem{benci1982oncritical}
Vieri Benci.
On critical point theory for indefinite functional in the presence of symmetries.
\emph{Trans. AMS} \textbf{274} (1982), no. 2, 533--572.

\bibitem{benci1979critical}
Vieri Benci, Paul H. Rabinowitz.
Critical points for indefinite functionals.
\emph{Invent. Math.} \textbf{52} (1979), 241--273.

\bibitem{bores1994harmonic}
Jarol\'im Bores.
\emph{Harmonic spinors on Riemann surfaces.} 
Proceedings of the Winter School "Geometry and
Physics". 
Circolo Matematico di Palermo, Palermo,
Rendiconti del Circolo Matematico di Palermo, Serie II, Supplemento No. 37: 15--32,
1994.

\bibitem{chai-lin-wang}
Ching-Li Chai, Chang-Shou Lin, Chin-Lung Wang. Mean field equations, hyperelliptic curves, and modular forms: I. \emph{Camb. J. Math.} \textbf{3} (2015), no. 1-2, 127--274.

\bibitem{changyang1987prescribing}
Sun-Yung Alice Chang, Paul C. Yang. 
Prescribing Gaussian curvature on~$S^2$.
\emph{Acta Math.} \textbf{159}, (1987), 215--259.

\bibitem{chen-kuo-lin}
Zhijie Chen, Ting-Jung Kuo, Chang-Shou Lin. Hamiltonian system for the elliptic form of Painlev\'e VI equation. \emph{J. Math. Pure App.} \textbf{106} (2016), no. 3, 546--581.

\bibitem{super2}
Takeshi Fukuda, Kazuo Hosomichi. Super-Liouville theory with boundary. \emph{Nucl. Phys. B} \textbf{635} (2002), 215--254.

\bibitem{ginoux2009dirac}
Nicolas Ginoux.
\emph{The Dirac Spectrum.}
Springer, Berlin,
2009.

\bibitem{gilbarg2001elliptic}
David Gilbarg, Neil S. Trudinger. 
\emph{Elliptic partial differential equations of second order}.
Springer-Verlag, Heidelberg,
2001. 

\bibitem{glockner2016fundamentals}
Helge Gl\"ockner.
Fundamentals of submersions and immersions between infinite-dimensional manifolds.
arXiv: 1502.05795v4[math.DG].

\bibitem{hitchin1974harmonicspinors}
Nigel Hitchin.
Harmonic spinors.
\emph{Adv. Math.} \textbf{14} (1974), 1--55.

\bibitem{hulshof1993differential}
Josephus Hulshof, Robertus Van der Vorst. 
Differential system with strongly indefinite variational structure.
\emph{J. Funct. Anal.} \textbf{114} (1993), 32--58.

\bibitem{isobe2010existence}
Takeshi Isobe.
Existence results for solutions to nonlinear Dirac equations on compact spin manifolds.
\emph{Manuscript Math.} \textbf{35} (2011), 329--360.

\bibitem{isobe2011nonlinear}
Takeshi Isobe.
Nonlinear dirac equations with critical nonlinearities on compact spin manifolds.
\emph{J. Funct. Anal.} \textbf{260} (2011), 253--307.

\bibitem{isobe2019onthemultiple}
Takeshi Isobe.
On the multiple existence of superquadratic dirac-harmonic maps into flat tori.
\emph{Calc. Var. PDEs}, (2019) 58:126.  

\bibitem{jost2007superLiouville}
J\"urgen Jost, Guofang Wang, Chunqin Zhou.
Super-Liouville equations on closed Riemann surfaces. 
\emph{Comm. PDEs} \textbf{32} (2007), no. 7, 1103--1128.

\bibitem{jost2009energy}
J\"urgen Jost, Guofang Wang, Chunqin Zhou, Miaomiao Zhu. 
Energy identities and blow-up analysis for solutions of the super-Liouville equation. 
 \emph{J. Math. Pures Appl.} \textbf{92} (2009), no. 3,  295--312.

\bibitem{jost2014qualitative}
J\"urgen Jost, Chunqin Zhou, Miaomiao Zhu. 
The qualitative boundary behavior of blow-up solutions of the super-Liouville equations. 
\emph{J. Math. Pures Appl.} \textbf{101} (2014), no. 5, 689--715. 


\bibitem{jost2015LocalEstimate}
J\"urgen Jost, Chunqin Zhou, Miaomiao Zhu.
A local estimate for super-Liouville equations on closed Riemann surfaces. 
 \emph{Calc. Var. PDEs}  \textbf{53} (2015), no. 1-2, 247--264.
 
\bibitem{jost2016regularity}
J\"urgen Jost, Enno Ke{\ss}ler, J\"urgen Tolksdorf, Ruijun Wu, Miaomiao Zhu.
Regularity of solutions of the nonlinear sigma model with gravitino.
\emph{Comm. Math. Phys.} \textbf{358} (2018) no. 1, 171--197. 
 
\bibitem{jost2018symmetries}
J\"urgen Jost, Enno Ke{\ss}ler, J\"urgen Tolksdorf, Ruijun Wu, Miaomiao Zhu.
Symmetries and conservation laws of a nonlinear sigma model with gravitino.
\emph{J. Geom. Phys.} \textbf{128} (2018), 185--198.

\bibitem{kazdan1975scalar} 
Jerry L. Kazdan, Frank W. Warner.
Curvature functions for compact 2-manifolds. \emph{Ann. Math.} \textbf{99} (1974), 14--47. 


\bibitem{lawson1989spin}
H. Blaine Lawson, Jr. ,Marie-Louise Michelsohn.
\emph{Spin geometry.}
Princeton Uni. Press,
1989.

\bibitem{maalaoui2013rabinowitz}
Ali Maalaoui.
Rabinowitz-Floer Homology for super-quadratic Dirac equations on spin manifolds.
\emph{J. Fixed Point Theory Appl.} \textbf{13} (2013), 175--199.

\bibitem{maalaoui2015therabinowitz}
Ali Maalaoui, Vittorio Martino.
The Rabinowitz--Floer homology for a class of semilinear problems and applications.
\emph{J. Funct. Anal.} \textbf{269} (2015), 4006--4037.

\bibitem{maalaoui2017characterization}
Ali Maalaoui, Vittorio Martino.
Characterization of the Palais-Smale sequences for the conformal Dirac-Einstein problem and applications.
\emph{J. Diff. Eq.} \textbf{266} (2019), 2493--2541.

\bibitem{ops1988extremals}  
Brad Osgood, Ralph Phillips, Peter Sarnak. 
Extremals of determinants of Laplacians. 
\emph{J. Funct. Anal.} \textbf{80} (1988), 1, 148--211. 

\bibitem{pankov2005periodic}
Alexander Pankov.
Periodic nonlinear Schr\"odinger equation with application to photonic crystal. 
\emph{Milan J. Math.} \textbf{73} (2005), 259--287.

\bibitem{polya1}
Aleksandr M. Polyakov. Quantum geometry of fermionic strings. \emph{Phys. Lett. B} \textbf{103} (1981), 207--210. 

\bibitem{polya2}
Aleksandr M. Polyakov. \emph{Gauge Fields and Strings}. Contemporary concepts in physics, Chur, Switzerland: Harwood Academic Publishers, 1987.

\bibitem{super3}
Jo\~ao N.G.N. Prata. The super-Liouville-equation on the half-line. \emph{Nucl. Phys. B} \textbf{405} (1997), 271--279.


\bibitem{schoenyau1994lectureOnDG-I}
Richard Schoen, Shing-Tung Yau.
\emph{Lectures on differential Geometry.} 
International Press of Boston, Inc,
1994.

\bibitem{hawking}
Yuguang Shi, Jiacheng Sun, Gang Tian, Dongyi Wei. Uniqueness of the mean field equation and rigidity of Hawking Mass. \emph{Calc. Var. and PDEs} (2019), 58:41.

\bibitem{spruck-yang}
Joel Spruck, Yisong Yang. On Multivortices in the Electroweak Theory I: Existence of Periodic Solutions. 
\emph{Comm. Math. Phys.} B (1992), 1--16.

\bibitem{struwe2008variational}
Micheal Struwe.
\emph{Variational methods: Applications to Nonlinear Partial Differential Equations and Hamiltonian Systems.} 
Springer-Verlag, Berlin Heidelberg,
2008.

\bibitem{szulkin2009ground}
Andrzej Szulkin, Tobias Weth.
Ground state solutions for some indefinite variational problems.
\emph{J. Funct. Anal.} \textbf{257} (2009), 3802--3822.

\bibitem{szulkin2010themethod}
Andrzej Szulkin, Tobias Weth.
The method of Nehari manifold.
\emph{Handbook of Nonconvex Analysis and Applications.}
Int. Press, Somerville, MA, (2010), 597--632.


\bibitem{tarantello}
Gabriella Tarantello. Multiple condensate solutions for the Chern-Simons-Higgs theory. 
\emph{J. Math. Phys.} \textbf{37} (1996), 3769--3796.

\bibitem{yang}
Yison Yang. \emph{Solitons in Field Theory and Nonlinear Analysis}. Springer Monographs in Mathematics, Springer, New York, 2001.

\end{thebibliography}
\end{document}